\UseRawInputEncoding
\documentclass[12pt]{amsart}
\usepackage{}
\usepackage{amsmath}
\usepackage{amsfonts}
\usepackage{amssymb}
\usepackage[all,cmtip]{xy}           
\usepackage{xspace}
\usepackage{bbding}
\usepackage{txfonts}
\usepackage{bbm}
\usepackage[shortlabels]{enumitem}
\usepackage{ifpdf}
\ifpdf
  \usepackage[colorlinks,final,backref=page,hyperindex]{hyperref}
\else
  \usepackage[colorlinks,final,backref=page,hyperindex,hypertex]{hyperref}
\fi
\usepackage{tikz}
\usepackage[active]{srcltx}

\makeatletter

\newtheorem{theorem}{Theorem}[section]
\newtheorem{prop}[theorem]{Proposition}
\newtheorem{lemma}[theorem]{Lemma}

\newtheorem{prop-def}{Proposition-Definition}[section]

\theoremstyle{definition}

\newtheorem{exam}[theorem]{Example}

\newcommand{\nc}{\newcommand}

\nc{\delete}[1]{{}}
\nc{\mmargin}[1]{}

\delete{
\nc{\mlabel}[1]{\label{#1}}  
\nc{\mcite}[1]{\cite{#1}}  
\nc{\mref}[1]{\ref{#1}}  
\nc{\meqref}[1]{\eqref{#1}} 
\nc{\mbibitem}[1]{\bibitem{#1}} 
}

\nc{\mlabel}[1]{\label{#1}  
{\hfill \hspace{1cm}{\bf{{\ }\hfill(#1)}}}}
\nc{\mcite}[1]{\cite{#1}{{\bf{{\ }(#1)}}}}  
\nc{\mref}[1]{\ref{#1}{{\bf{{\ }(#1)}}}}  
\nc{\meqref}[1]{\eqref{#1}{{\bf{{\ }(#1)}}}} 
\nc{\mbibitem}[1]{\bibitem[\bf #1]{#1}} 

\nc{\shadow}{phantom\xspace}
\nc{\Shadow}{Phantom\xspace}
\nc{\shad}{\theta}
\nc{\tforall}{\text{ for all }}
\nc{\ddiff}{D-\text{differential}\xspace}
\nc{\oplin}{operator linear\xspace}
\nc{\fopav}{F_A(\Omega,V)}
\nc{\wtd}{weighted\xspace}
\nc{\wt}{weight\xspace}
\nc{\wte}{\lambda}
\nc{\coa}{\overline{\mathfrak{OA}}} 
\nc{\oa}{\mathfrak{OA}} 

\nc{\lc}{\lfloor} \nc{\rc}{\rfloor}
\nc{\free}[1]{\widetilde{#1}}
\nc{\Id}{\mathrm{Id}}
\nc{\lra}{\longrightarrow}

\nc{\vep}{\varepsilon}
\nc{\bin}[2]{ (_{\stackrel{\scs{#1}}{\scs{#2}}})}  
\nc{\binc}[2]{(\!\! \begin{array}{c} \scs{#1}\\
    \scs{#2} \end{array}\!\!)}  
\nc{\bincc}[2]{  ( {\scs{#1} \atop
    \vspace{-1cm}\scs{#2}} )}  
\nc{\bs}{\bar{S}}
\nc{\ra}{\longleftarrow}
\nc{\ot}{\otimes}
\nc{\rar}{\rightarrow}
\nc{\dar}{\downarrow}
\nc{\dap}[1]{\downarrow \rlap{$\scriptstyle{#1}$}}
\nc{\defeq}{\stackrel{\rm def}{=}}
\nc{\dis}[1]{\displaystyle{#1}}
\nc{\dotcup}{\ \displaystyle{\bigcup^\bullet}\ }
\nc{\hcm}{\ \hat{,}\ }
\nc{\hts}{\hat{\otimes}}
\nc{\hcirc}{\hat{\circ}}
\nc{\lleft}{[}
\nc{\lright}{]}
\nc{\curlyl}{\left \{ \begin{array}{c} {} \\ {} \end{array}
    \right .  \!\!\!\!\!\!\!}
\nc{\curlyr}{ \!\!\!\!\!\!\!
    \left . \begin{array}{c} {} \\ {} \end{array}
    \right \} }
\nc{\longmid}{\left | \begin{array}{c} {} \\ {} \end{array}
    \right . \!\!\!\!\!\!\!}
\nc{\ora}[1]{\stackrel{#1}{\rar}}
\nc{\ola}[1]{\stackrel{#1}{\la}}
\nc{\scs}[1]{\scriptstyle{#1}}
\nc{\mrm}[1]{{\rm #1}}
\nc{\dirlim}{\displaystyle{\lim_{\longrightarrow}}\,}
\nc{\invlim}{\displaystyle{\lim_{\longleftarrow}}\,}
\nc{\dislim}[1]{\displaystyle{\lim_{#1}}} \nc{\colim}{\mrm{colim}}
\nc{\mvp}{\vspace{0.3cm}} \nc{\tk}{^{(k)}} \nc{\tp}{^\prime}
\nc{\ttp}{''} \nc{\svp}{\vspace{2cm}}
\nc{\vp}{\vspace{8cm}}
\nc{\modg}[1]{\!<\!\!{#1}\!\!>}
\nc{\intg}[1]{F_C(#1)}
\nc{\lmodg}{\!<\!\!}
\nc{\rmodg}{\!\!>\!}
\nc{\cpi}{\widehat{\Pi}}
\nc{\sha}{{\mbox{\cyr X}}}  
\nc{\ssha}{{\mbox{\cyrs X}}} 
\nc{\tsha}{{\mbox{\cyrt X}}}
\nc{\shai}{{\stackrel{\ra}{\sha}}}
\nc{\shpr}{\diamond}    
\nc{\labs}{\mid\!}
\nc{\rabs}{\!\mid}

\nc{\dr}{\frakR}
\nc{\cdr}{\frakC\frakR}
\nc{\er}{\frakW}
\nc{\cer}{\frakC\frakW}
\nc{\DR}{\overline{\frakR}}
\nc{\CDR}{\overline{\frakC\frakR}}
\nc{\ER}{\overline{\frakW}}
\nc{\CER}{\overline{\frakC\frakW}}

\font\cyr=wncyr10
\font\cyrs=wncyr7
\font\cyrt=wncyr5

\nc{\ann}{\mrm{ann}}
\nc{\Aut}{\mrm{Aut}}
\nc{\can}{\mrm{can}}
\nc{\Cont}{\mrm{Cont}}
\nc{\rchar}{\mrm{char}}
\nc{\cok}{\mrm{coker}}
\nc{\dtf}{{R-{\rm tf}}}
\nc{\dtor}{{R-{\rm tor}}}

\nc{\Div}{{\mrm Div}}
\nc{\End}{\mrm{End}}
\nc{\Ext}{\mrm{Ext}}
\nc{\Fil}{\mrm{Fil}}
\nc{\Fr}{\mrm{Fr}}
\nc{\Frob}{\mrm{Frob}}
\nc{\Gal}{\mrm{Gal}}
\nc{\GL}{\mrm{GL}}
\nc{\Hom}{\mrm{Hom}}
\nc{\hsr}{\mrm{H}}
\nc{\hpol}{\mrm{HP}}
\nc{\id}{\mrm{id}}
\nc{\im}{\mrm{im}}
\nc{\incl}{\mrm{incl}}
\nc{\Irr}{\mrm{Irr}}
\nc{\length}{\mrm{length}}
\nc{\mforall}{\quad \text{for all }}
\nc{\mchar}{\rm char}
\nc{\mpart}{\mrm{part}}
\nc{\ql}{{\QQ_\ell}}
\nc{\qp}{{\QQ_p}}
\nc{\rank}{\mrm{rank}}
\nc{\rcot}{\mrm{cot}}
\nc{\rdef}{\mrm{def}}
\nc{\rdiv}{{\rm div}}
\nc{\rtf}{{\rm tf}}
\nc{\rtor}{{\rm tor}}
\nc{\res}{\mrm{res}}
\nc{\SL}{\mrm{SL}}
\nc{\Spec}{\mrm{Spec}}
\nc{\tor}{\mrm{tor}}
\nc{\Tr}{\mrm{Tr}}
\nc{\tr}{\mrm{tr}}

\nc{\bfk}{{\bf k}}
\nc{\bfone}{{\bf 1}}
\nc{\bfzero}{{\bf 0}}
\nc{\Diff}{\mathbf{Diff}}
\nc{\FMod}{\mathbf{FMod}}
\nc{\Int}{\mathbf{Int}}
\nc{\Mon}{\mathbf{Mon}}
\nc{\remarks}{\noindent{\bf Remarks: }}
\nc{\Rep}{\mathbf{Rep}}
\nc{\Rings}{\mathbf{Rings}}
\nc{\Sets}{\mathbf{Sets}}

\nc{\BA}{{\mathbb A}}   \nc{\CC}{{\mathbb C}}
\nc{\DD}{{\mathbb D}}   \nc{\EE}{{\mathbb E}}
\nc{\FF}{{\mathbb F}}   \nc{\GG}{{\mathbb G}}
\nc{\HH}{{\mathbb H}}   \nc{\LL}{{\mathbb L}}
\nc{\NN}{{\mathbb N}}   \nc{\PP}{{\mathbb P}}
\nc{\QQ}{{\mathbb Q}}   \nc{\RR}{{\mathbb R}}
\nc{\TT}{{\mathbb T}}   \nc{\VV}{{\mathbb V}}
\nc{\ZZ}{{\mathbb Z}}   \nc{\TP}{\widetilde{P}}

\nc{\cala}{{\mathcal A}}    \nc{\calc}{{\mathcal C}}
\nc{\cald}{\mathcal{D}}     \nc{\cale}{{\mathcal E}}
\nc{\calf}{{\mathcal F}}    \nc{\calg}{{\mathcal G}}
\nc{\calh}{{\mathcal H}}    \nc{\cali}{{\mathcal I}}
\nc{\call}{{\mathcal L}}    \nc{\calm}{{\mathcal M}}
\nc{\caln}{{\mathcal N}}    \nc{\calo}{{\mathcal O}}
\nc{\calp}{{\mathcal P}}    \nc{\calr}{{\mathcal R}}
\nc{\cals}{{\mathcal S}}    \nc{\calt}{{\Omega}}
\nc{\calw}{{\mathcal W}}    \nc{\calx}{{\mathcal X}}
\nc{\CA}{\mathcal{A}}

\nc{\fraka}{{\mathfrak a}}
\nc{\frakb}{{\mathfrak b}}
\nc{\frakc}{{\mathfrak c}}
\nc{\frakd}{{\mathfrak d}}
\nc{\frake}{{\mathfrak e}}
\nc{\fraki}{{\mathfrak i}}
\nc{\frakj}{{\mathfrak j}}
\nc{\frakk}{{\mathfrak k}}
\nc{\frakB}{{\frak B}}
\nc{\frakC}{{\frak C}}
\nc{\frakE}{{\frak E}}
\nc{\frakm}{{\frak m}}
\nc{\frakM}{{\frak M}}
\nc{\frakp}{{\frak p}}
\nc{\frakR}{{\frak R}}
\nc{\frakS}{{\frak S}}
\nc{\frakA}{{\frak A}}
\nc{\frakW}{{\frak W}}
\nc{\fraku}{{\frak u}}
\nc{\frakv}{{\frak v}}
\nc{\frakw}{{\frak w}}
\nc{\frakx}{{\frak x}}
\nc{\fraky}{{\frak y}}

\nc{\ynr}[1]{\textcolor{blue}{\underline{Yunnan:}#1 }}

\nc{\lir}[1]{\textcolor{red}{\underline{Li:}#1 }}

\nc{\rgr}[1]{\textcolor{violet}{\underline{Richard:}#1 }}

\begin{document}

\title[Primitive decompositions of idempotents]{Primitive decompositions of idempotents of the group algebras of dihedral groups and generalized quaternion groups}

\author{Lilan Dai}
\address{School of Mathematics and Information Science, Guangzhou University, Waihuan Road West 230, Guangzhou 510006, China}
\email{2112015048@e.gzhu.edu.cn}

\author{Yunnan Li}
\address{School of Mathematics and Information Science, Guangzhou University,
Waihuan Road West 230, Guangzhou 510006, China}
\email{ynli@gzhu.edu.cn}

\date{\today}

\begin{abstract}
In this paper, we introduce a method computing the primitive decomposition of idempotents of any semisimple finite group algebra based on its matrix representations and Wedderburn decomposition. Particularly, we use this method to calculate the examples of the dihedral group algebras $\mathbb{C}[D_{2n}]$ and generalized quaternion group algebras $\mathbb{C}[Q_{4m}]$. Inspired by the orthogonality relations of the character tables of these two families of groups, we obtain two sets of trigonometric identities. Furthermore, a group algebra isomorphism between $\mathbb{C}[D_{8}]$ and $\mathbb{C}[Q_{8}]$ is described, under which the two complete sets of primitive orthogonal idempotents of these two group algebras we find correspond to each other bijectively.
\end{abstract}

\subjclass[2010]{20C05, 20C15}

\keywords{idempotent, primitive decomposition, group algebra}

\maketitle

\tableofcontents

\allowdisplaybreaks

\section{Introduction}
Given any finite group $G$ and field $F$, denote $F[G]$ the group ring of $G$ over $F$. When ${\rm char}\,F\nmid|G|$, $F[G]$ is semisimple by Maschke's theorem. Then by Wedderburn's structure theorem, $F[G]$ is isomorphic to a direct sum of matrix algebras. The Wedderburn decomposition becomes a key tool for studying group algebra problems (\cite{FF,He,JeL,JR1,ORS2,RS}). For example,  Macedo Ferreira et al. dealt with the Wedderburn $b$-decomposition for alternative baric algebras \cite{FF}. Jespers et al. reduced the number of generators for a subgroup of finite index in a certain kind of unit group $\mathcal{U}(\mathbb{Z}[G])$ by having a closer look at the Wedderburn decomposition of $\mathbb{Q}[G]$ \cite{JeL}. Olivieri et al. studied the automorphism group $\Aut(\mathbb{Q}[G])$ of the rational group algebra $\mathbb{Q}[G]$ of a finite metacyclic group $G$ by describing the simple components of the Wedderburn decomposition of $\mathbb{Q}[G]$ \cite{ORS2}.

As the main objects discussed throughout our paper, dihedral groups $D_{2n}$ describe 2-dimensional objects that have rotational and reflective symmetry, such as regular polygons, and generalized quaternion groups $Q_{4m}$ generalize the quaternion group $Q_8$.
In physics, the theory of rigid motion analysis and the practical problem of motion control are all related to quaternions, and many applications in physics use the concept and extension of quaternions.

The Wedderburn decomposition of group algebras of these two families of groups has already attracted much attention. For instance, Giraldo Vergara and Brochero Mart\'inez gave an elementary proof of the Wedderburn decomposition of rational quaternion and dihedral group algebras \cite{VM}. Giraldo Vergara used the classification of groups of order $\leq32$ and also computed the Wedderburn decomposition of their rational group algebras in order to classify the rational group algebras of dimension $\leq32$ \cite{Ve}. Bakshi et al. calculated a complete set of primitive central idempotents and the Wedderburn decomposition of the rational group algebra of a finite metabelian group \cite{BKP}. Brochero Mart\'inez showed explicitly the primitive central idempotents of $F_{q}[D_{2n}]$ and an isomorphism between the group algebra $F_{q}[D_{2n}]$ and its Wedderburn decomposition when every prime factor of $n$ divides $q-1$ \cite{Ma}. Gao and Yue focused on the algebraic structure of the generalized quaternion group algebras $F_{q}[Q_{4m}]$ over finite field $F_{q}$~\cite{GY}.

What's more, the study on primitive orthogonal idempotents of group algebras has ignited much interest.
For many classes of groups, such as nilpotent, monomial and supersolvable groups, a complete description of the idempotents of their group algebras was obtained by Berman~(see e.g. \cite{ILPSSZ}). For example, Berman actually has constructed the minimal central idempotents of the group ring $R(G, F)$ in terms of the central idempotents of $R(H, F)$ when $G$ is an abelian extension of a group $H$ in 1955. Furthermore, the complete system of minimal idempotents of $R(G,F)$ was given in terms of such a system for $R(H, F)$ when $G/H$ is cyclic~\cite{B1}. After that, he characterized a complete system of primitive orthogonal idempotents of $F[G]$ for any solvable group $G$ of class $M_1$ in~\cite{B2} by calculating linear characters of its subgroups, where $F$ is any field of characteristic prime to $|G|$ containing a primitive root of unity of $|G|$.

After nearly 40 years, a method somewhat different but closely related to Berman's in calculating primitive orthogonal idempotents of these group algebras was proposed.
Around 2004, Olivieri et al. gave a character-free method to describe the primitive central idempotents of $\mathbb{Q}[G]$ when $G$ is a monomial group~\cite{ORS1}.
Later, an explicit and character-free construction of a complete set of primitive orthogonal idempotents of $\mathbb{Q}[G]$ was provided in~\cite{JOR} for any finite nilpotent group $G$ (see also~\cite{OG} for the case over finite fields), and in~\cite{JORV} for any finite strongly monomial group $G$ such that there exists a complete and non-redundant set of strong Shoda pairs with trivial twistings. See also \cite[Chapter 13]{JR2} for an overall introduction to this topic.

In this paper, after calculating the primitive central idempotents of $\mathbb{C}[D_{2n}]$ and $\mathbb{C}[Q_{4m}]$ via irreducible characters,
we further consider their primitive decompositions of idempotents. Note that dihedral groups $D_{2n}$ and generalized quaternion groups $Q_{4m}$ are not only supersolvable groups, but also strongly monomial groups. Their primitive decompositions of idempotents surely can be obtained by Berman's method in \cite{B2}. Also, a complete set of primitive orthogonal idempotents of any dihedral group can be constructed via strong Shoda pairs, but questionably for all generalized quaternion groups~\cite[\S~4]{JORV}. By contrast, the computation of primitive decompositions of idempotents here mainly depends on matrix representations of groups and Wedderburn decompositions of group algebras~(Lemma~\ref{lem:wed}). Such an approach is theoretically applicable to any semisimple group algebra over arbitrary field whenever a complete set of its non-equivalent irreducible matrix representations have been figured out. In particular, it's exactly available to the examples of dihedral groups and generalized quaternion groups.

On the other hand, given two primitive decomposition of idempotents of two isomorphic group algebras respectively, it seems very difficult to obtain
 a certain algebra isomorphism between them, making these two complete sets of primitive orthogonal idempotents correspond to each other.
 Here we solve one small but nontrivial case, establishing an explicit isomorphism between $\mathbb{C}[D_{8}]$ and $\mathbb{C}[Q_{8}]$ which respects the list of primitive orthogonal idempotents we find previously. Indeed, there already have plenty of results for the group algebras of $D_{8}$ and $Q_{8}$. For example, Bagi\'nski studied group algebras of 2-groups of maximal class over fields of characteristic 2, so we know that $F_{2}[D_{8}]$ and $F_{2}[Q_{8}]$ are not isomorphic as rings~\cite{CB}.
Coleman discussed group rings over the complex and real number fields and over the ring of integers in \cite{Co}, which tells us that $\mathbb{C}[Q_{8}]\cong \mathbb{C}[D_{8}]$, but $\mathbb{R}[Q_{8}]\ncong \mathbb{R}[D_{8}]$ and $\mathbb{Z}[Q_{8}]\ncong \mathbb{Z}[D_{8}]$. As $\mathbb{R}$ is a field extension of $\mathbb{Q}$, it also implies that $\mathbb{Q}[Q_{8}]\ncong\mathbb{Q}[D_{8}]$ .
Tambara and Yamagami pointed out that $Q_{8}$ and $D_{8}$ have the same representation ring, but non-isomorphic representation categories as tensor categories \cite{TY}.

Here is the layout of the paper. In Section~\ref{se:di-idem} and Section~\ref{se:gqg-idem}, the primitive central idempotents of dihedral groups and generalized quaternion groups are calculated by  their irreducible characters.
Furthermore, primitive decompositions of idempotents corresponding to their two-dimensional representations are analyzed. In Section~\ref{se:tri-iden}, two sets of general trigonometric identities reflecting the orthogonality relations of irreducible characters of dihedral groups and generalized quaternion groups are given. In Section~\ref{se:alg-isom}, a group algebra isomorphism between $\mathbb{C}[Q_{8}]$ and $\mathbb{C}[D_{8}]$ is described, which also provides a correspondence between their primitive orthogonal idempotents previously worked out.

\section{A primitive decomposition of idempotents of $\CC[D_{2n}]$}\label{se:di-idem}
\subsection{Conjugacy classes of $D_{2n}$}
 Let $D_{2n}$ be the dihedral group of order $2n$, i.e.
 $$D_{2n}=\{r,s\,|\,r^n=s^2={\bf 1},\,srs=r^{-1}\}=\{{\bf 1},r,\dots,r^{n-1},s,rs,\dots,r^{n-1}s\}.$$
When $n$ is odd, namely $n=2m+1$, $D_{2n}$ has the following conjugacy classes:
$$[{\bf 1}]=\{{\bf 1}\},\,[r^i]=\{r^{\pm i}\,|\,1 \leq i \leq m\},\,[s]=\{s,rs,\dots,r^{n-1}s\}.$$
When $n$ is even, namely $n=2m$, $D_{2n}$ has the following conjugacy classes:
$$\{{\bf1}\},\,\{r^{m}\},\,\{r^{\pm i}\,|\,1\leq i\leq m-1\},\,\{r^{2k}s\,|\,0\leq k \leq m-1\},\,\{r^{2k+1}s\,|\,0\leq k \leq m-1\}.$$
\subsection{Character table of $D_{2n}$}\label{subsec:char}
(i) $n=2m+1$. We look at the one-dimensional representations first. Note that $D_{2n}/\langle r\rangle \cong \langle s \rangle$, which is abelian, hence the derived subgroup $D'_{2n}\subseteq \langle r\rangle$. Clearly, $s^{-1}r^{-1}sr=r^2\in D'_{2n}$, thus we have $D'_{2n}\supseteq \langle r^2\rangle$. Note that $r^{2m}=r^{-1}\in \langle r^2\rangle$, therefore $\langle r^2\rangle=\langle r\rangle$. Then $D'_{2n}=\langle r\rangle$. As a result, $D_{2n}$ has two one-dimensional representations and $D_{2n}/\langle r\rangle \cong C_{2}$, where $C_{2}$ is the cyclic group of order 2.

Next we introduce these two-dimensional irreducible representations of $D_{2n}$ from its natural geometric description~ \cite[Part I, 5.3]{S}. We can set up a rectangular coordinate system, where the origin is the center of a regular $n$-sided polygon, and the angular bisectors in the first and third quadrants is one of the symmetry axes of the regular $n$-sided polygon. Since $D_{2n}$ is a permutation group of regular $n$-sided polygons, the matrices of $r$, $s$ with respect to the standard basis can be given. Then we have the following natural representations:
\begin{equation}\label{eq:di-1-repn}
\rho_{k}(r)=\begin{pmatrix}
      \cos\frac{2k\pi}{n} &-\sin\frac{2k\pi}{n}\\[.2em]
      \sin\frac{2k\pi}{n}&\cos\frac{2k\pi}{n}
      \end{pmatrix},\quad
 \rho_{k}(s)=\begin{pmatrix}
      0 &1\\
      1&0
      \end{pmatrix},\quad  1 \leq k \leq m.
\end{equation}
which are $m$ mutually non-equivalent two-dimensional irreducible representations of $D_{2n}$. Thus, when $n $ is an odd number, we set $\theta = \dfrac {2\pi}{n} $, and list the character table of $D_{2n}$:
\begin{table}[h]
\begin{tabular}{|c|c|c|c|c|c|c|c|c|}

     \hline
   & {\bf1} & $s$ & $r$ & $r^2$ &$r^3$& $\cdots$ & $r^{m-1}$ &$r^m$ \\
   & (1)& ($n$) & (2) & (2) & (2)&$\cdots$ & (2)&(2) \\
  \hline
  $\chi_{1} $& 1 & 1 & 1 & 1 &1& $\cdots$ & 1 &1\\
  $\chi_{2}$ & 1 & $-1$ & 1 & 1 &1& $\cdots$ & 1 &1\\
  $\chi_{\rho_{1}}$ & 2 & 0 & $2\cos\theta$ & $2\cos2\theta$ & $2\cos3\theta$ &$\cdots$ &$2\cos (m-1)\theta$& $2\cos m\theta$ \\
  $\vdots$ & $\vdots$ & $\vdots$ & $\vdots$ & $\vdots$ & $\vdots$ & $\vdots$ & $\vdots$ & $\vdots$ \\
  $\chi_{\rho_{m}}$ & 2 & 0 & $2\cos m\theta$ & $2\cos2m\theta$ & $2\cos3m\theta$ & $\cdots$ &$2\cos (m-1)m\theta$& $2\cos m^2\theta$ \\
  \hline
\end{tabular}
\vspace{.5em}
\caption{Irreducible characters of $D_{2n}$}\label{tb:1}
\end{table}

(ii) $n=2m$. Similarly, $\langle r^2\rangle$ is a normal subgroup of $D_{2n}$ as $sr^2s^{-1}=r^{-2}\in\langle r^2\rangle$, and $|D_{2n}/\langle r^2\rangle|=4$, then $D_{2n}/\langle r^2\rangle$ is abelian, and thus $D'_{2n}\subseteq\langle r^2\rangle$. Clearly, $r^2=s^{-1}r^{-1}sr\in D'_{2n}$, we also have $D'_{2n}\supseteq\langle r^2\rangle$, so $D'_{2n}=\langle r^2\rangle$. As a result, $D_{2n}$ has four one-dimensional representations and $D_{2n}/\langle r^2\rangle\cong C_{2}\times C_{2}$.

 If $n$ is an even number, we can also obtain $m-1$ pairwise non-equivalent two-dimensional irreducible representations of $D_{2n}$:
\begin{equation}\label{eq:di-2-repn}
\rho_{k}(r)=\begin{pmatrix}
      \cos\frac{2k\pi}{n} &-\sin\frac{2k\pi}{n}\\[.2em]
      \sin\frac{2k\pi}{n}&\cos\frac{2k\pi}{n}
      \end{pmatrix},\quad
 \rho_{k}(s)=\begin{pmatrix}
      0 &1\\
      1&0
      \end{pmatrix},\quad  1 \leq k \leq m-1.
\end{equation}
  Thus, when $n $ is an even number, we set $\theta = \dfrac {2\pi}{n} $, and list the character table of $D_{2n}$:
 \begin{table}[h]
 \begin{tabular}{|c|c|c|c|c|c|c|c|c|}
    \hline
   & {\bf1} & $s$ & $sr$ & $r$ & $r^2$ & \dots & $r^{m-1}$ & $r^m$ \\
   & (1)& ($m$) & ($m$) & (2) & (2) & \dots & (2)& (1) \\
  \hline
  $\chi_{1} $& 1 & 1& 1 & 1 & 1 & \dots & 1& 1 \\
  $\chi_{2}$ & 1 & 1& $-1$ & $-1$ & 1 & $\cdots$  & $(-1)^{m-1}$ & $(-1)^{m}$ \\
  $\chi_{3}$ & 1 & $-1$& $1$ & $-1$ & 1 & $\cdots$  & $(-1)^{m-1}$ & $(-1)^{m}$ \\
  $\chi_{4}$ & 1 & $-1$& $-1$ & $1$ & 1 & $\cdots$  & $1$& $1$ \\
  $\chi_{\rho_{1}}$ & 2 & 0& 0 & $2\cos\theta$ & $2\cos2\theta$ & $\cdots$  & $2\cos(m-1)\theta$ &$-2$ \\
  $\vdots$ &   $\vdots$ &   $\vdots$ &   $\vdots$ &   $\vdots$ &   $\vdots$ &   $\vdots$ & $\vdots$ &   $\vdots$ \\
  $\chi_{\rho_{m-1}}$ & 2 & 0& 0 & $2\cos(m-1)\theta$ & $2\cos2(m-1)\theta$ & $\cdots$  & $2\cos(m-1)^{2}\theta$ &$2(-1)^{m-1}$\\
  \hline
\end{tabular}
\vspace{.5em}
\caption{Irreducible characters of $D_{2n}$}\label{tb:2}
\end{table}
\subsection{A primitive decomposition of idempotents}
\begin{theorem}(Wedderburn Structure Theorem).\label{thn:wed}
Let $F$ be any field such that ${\rm char}\,F\nmid|G|$. Then
$$F[G]\stackrel{\varphi}{\cong} M_{n_{1}}(D_1)\oplus \cdots\oplus M_{n_{s}}(D_s)$$
as algebras, where $D_k$ is a division $F$-algebra, and each matrix algebra $M_{n_{k}}(D_k)$ uniquely determines an irreducible representation $\rho_{k}$ of $G$ up to equivalence, and $n_{k}$ is equal to its dimension over $D_k$ for $k=1,\dots,s$.
\end{theorem}
According to Theorem~\ref{thn:wed}, we obtain the following useful lemma.
\begin{lemma}\label{lem:wed}
For any semisimple finite group algebra $F[G]$, let $e_{\rho_k}$ be the primitive central idempotent of $F[G]$ corresponding to $\rho_{k}$. The group homomorphism $\rho_k:G\to GL(n_k,D_k)$ can be linearly extended to the following algebra homomorphism
$$F[G]\stackrel{\varphi}{\cong}  M_{n_{1}}(D_1)\oplus \cdots\oplus M_{n_{s}}(D_s)\stackrel{p_k}{\to} M_{n_{k}}(D_k),$$
which is an isomorphism when restricted on $F[G]e_{\rho_k}$. In particular, the preimages of the matrix units $E_{11},\dots,E_{n_k,n_k}$ of $ M_{n_{k}}(D_k)$ under this isomorphism provide a primitive decomposition of $e_{\rho_k}$ in $F[G]$.
Here we denote $p_k$ the natural projection.
\end{lemma}
Also, it is well-known that all primitive central idempotents of the semisimple  group algebra $F[G]$ of a finite group $G$ can be obtained by its character table (see e.g. \cite[Theorem 3.6.2]{We}), namely
\begin{equation}\label{eq:pcidem}
e_{\chi}=\frac{1}{|G|}\sum\limits_{g\in G} \chi({\bf1})\chi(g^{-1})g,\quad\forall \chi\in \Irr(G).
\end{equation}

Applying Eq.~\eqref{eq:pcidem} to the case of dihedral group $D_{2n}$, we immediately have
\begin{prop}\label{prop:di-idem}
Let $D_{2n}$ be the dihedral group of order $2n$. The primitive central idempotents corresponding to the one-dimensional irreducible representations of $D_{2n}$ are as follows.
\begin{enumerate}[(i)]
\item
When $n$ is odd, namely $n=2m+1$,
\begin{align*}
e_{1}&=\frac{1}{4m+2}(\sum\limits_{l=1}^{2m+1} r^l+\sum\limits_{l=1}^{2m+1} r^ls),\\
e_{2}&=\frac{1}{4m+2}(\sum_{l=1}^{2m+1} r^l-\sum_{l=1}^{2m+1}r^ls).\\
\end{align*}
\item
When $n$ is even, namely $n=2m$,
\begin{align*}
e_{1}&=\frac{1}{4m}(\sum\limits_{l=1}^{2m} r^l+\sum\limits_{l=1}^{2m} r^ls),\\
e_{2}&=\frac{1}{4m}[{\bf 1}+\sum_{l=1}^{2m}(-1)^l\cdot r^ls+\sum_{l=1}^{m-1}(-1)^l\cdot (r^l+r^{-l})+(-1)^m\cdot r^m],\\
e_{3}&=\frac{1}{4m}[{\bf 1}+\sum_{l=1}^{2m}(-1)^{l+1}\cdot r^ls+\sum_{l=1}^{m-1}(-1)^l\cdot (r^l+r^{-l})+(-1)^m\cdot r^m],\\
e_{4}&=\frac{1}{4m}(\sum\limits_{l=1}^{2m} r^l-\sum\limits_{l=1}^{2m} r^ls).
\end{align*}
\end{enumerate}
\end{prop}

In order to obtain a primitive decomposition of idempotents of $\CC[D_{2n}]$, we mainly need to deal with its primitive idempotents corresponding to two-dimensional irreducible representations.
\begin{theorem}\label{th:di-idem}
Let $D_{2n}$ be the dihedral group of order $2n$. We have the following primitive decomposition $e_{\rho_{k}}=e'_{\rho_{k}}+e''_{\rho_{k}}$ of the primitive central idempotent $e_{\rho_{k}}$ corresponding to the two-dimensional irreducible representation $(\CC^2,\rho_k)$ of $D_{2n}$ defined in Eqs.~\eqref{eq:di-1-repn} and \eqref{eq:di-2-repn} for $k=1,\dots,\lfloor(n-1)/2\rfloor$.
\begin{enumerate}[(i)]
\item
When $n$ is odd, namely $n=2m+1$,
\begin{align*}
e_{\rho_{k}}&=\frac{2}{2m+1}\sum\limits_{l=1}^{2m+1}\cos lk\theta\cdot r^l,\\
e'_{\rho_{k}}&=\frac{1}{2m+1}({\bf 1}+\sum_{l=1}^{2m}\cos lk\theta \cdot r^l+\sum_{l=1}^{2m}\sin lk\theta\cdot r^ls),\\
e''_{\rho_{k}}&=\frac{1}{2m+1}({\bf 1}+\sum\limits_{l=1}^{2m}\cos lk\theta \cdot r^l-\sum\limits_{l=1}^{2m}\sin lk\theta\cdot r^ls),
\end{align*}
with $\theta = \dfrac {2\pi}{n}$ and $1\leq k \leq m$;
\item
When $n$ is even, namely $n=2m$,
\begin{align*}
e_{\rho_{k}}&=\frac{1}{m}\sum_{l=1}^{2m}\cos lk\theta \cdot r^l,\\
e'_{\rho_{k}}&=\frac{1}{2m}({\bf 1}+\sum_{l=1}^{2m-1}\cos lk\theta \cdot r^l+\sum_{l=1}^{2m-1}\sin lk\theta\cdot r^ls),\\
e''_{\rho_{k}}&=\frac{1}{2m}({\bf 1}+\sum\limits_{l=1}^{2m-1}\cos lk\theta \cdot r^l-\sum\limits_{l=1}^{2m-1}\sin lk\theta\cdot r^ls),
\end{align*}
with $\theta = \dfrac {2\pi}{n}$ and $1\leq k \leq m-1$.
\end{enumerate}
\end{theorem}
\begin{proof}
Under the group homomorphism $\rho_k:D_{2n}\to GL(2,\CC)$, we have
$$r \mapsto \begin{pmatrix}
      \cos\frac{2k\pi}{n} &-\sin\frac{2k\pi}{n}\\[.2em]
      \sin\frac{2k\pi}{n}&\cos\frac{2k\pi}{n}
     \end{pmatrix},\quad
      s \mapsto \begin{pmatrix}
      0 &1\\
      1&0
      \end{pmatrix},\quad
      {\bf 1}\mapsto \begin{pmatrix}
      1 &0\\
      0&1
      \end{pmatrix}.$$
Therefore,
$$rs\mapsto \begin{pmatrix}
      -\sin\frac{2k\pi}{n} &\cos\frac{2k\pi}{n}\\[.2em]
      \cos\frac{2k\pi}{n}&\sin\frac{2k\pi}{n}
      \end{pmatrix},\quad
      \cos\frac{2k\pi}{n}s-rs \mapsto \begin{pmatrix}
      \sin\frac{2k\pi}{n} &0\\
      0&-\sin\frac{2k\pi}{n}
      \end{pmatrix}.$$
Thus,
$$\sin\frac{2k\pi}{n}{\bf 1}-(\cos\frac{2k\pi}{n}s-rs)\mapsto\begin{pmatrix}
        0&0\\
      0&2\sin\frac{2k\pi}{n}
      \end{pmatrix},$$
$$\sin\frac{2k\pi}{n}{\bf 1}+(\cos\frac{2k\pi}{n}s-rs)\mapsto\begin{pmatrix}
      2\sin\frac{2k\pi}{n} &0\\
      0&0
      \end{pmatrix}.$$
      Clearly, $0<\dfrac{2k\pi}{n}<\pi$, we have
$$\frac{1}{2\sin\frac{2k\pi}{n}}(\sin\frac{2k\pi}{n}{\bf 1}-\cos\frac{2k\pi}{n}s+rs)\mapsto\begin{pmatrix}
      0 &0\\
      0&1
      \end{pmatrix},$$
$$\frac{1}{2\sin\frac{2k\pi}{n}}(\sin\frac{2k\pi}{n}{\bf 1}+\cos\frac{2k\pi}{n}s-rs)\mapsto\begin{pmatrix}
      1 &0\\
      0&0
      \end{pmatrix}.$$
By Lemma~\ref{lem:wed}, we know that $F[G]e_{\rho_{k}}\cong M_{n_{k}}(F)$ as algebras, and thus
\begin{align*}
        e'_{\rho_{k}} & =e_{\rho_{k}}\cdot \frac{1}{2\sin\frac{2k\pi}{n}}(\sin\frac{2k\pi}{n}{\bf 1}-\cos\frac{2k\pi}{n}\cdot s+rs)\\
        &=\frac{1}{2}e_{\rho_{k}}\cdot ({\bf 1}-\cot k\theta\cdot s+\csc k\theta\cdot rs), \\
        e''_{\rho_{k}}&=\frac{1}{2}e_{\rho_{k}}\cdot ({\bf 1}+\cot k\theta\cdot s-\csc k\theta\cdot rs).
      \end{align*}
     We can verify that
     $$e_{\rho_{k}}=e'_{\rho_{k}}+e''_{\rho_{k}},\quad e'_{\rho_{k}}\cdot  e''_{\rho_{k}}=0,\quad  e'_{\rho_{k}}\cdot  e'_{\rho_{k}}=e'_{\rho_{k}},\quad e''_{\rho_{k}}\cdot  e''_{\rho_{k}}=e''_{\rho_{k}}.$$

(i) If $n=2m+1$, the primitive central idempotents $e_{\rho_k}$ are given as follows by Eq.~\eqref{eq:pcidem} and the character table of $D_{2n}$:
$$e_{\rho_{k}}
=\frac{2}{2m+1}\sum\limits_{l=1}^{2m+1}\cos lk\theta\cdot r^l,\quad 1\leq k \leq m.$$
Thus
$$e'_{\rho_{k}}=\frac{1}{2m+1}({\bf 1}+\sum_{l=1}^{2m}\cos lk\theta \cdot r^l+\sum_{l=1}^{2m}\sin lk\theta\cdot r^ls),\quad 1\leq k \leq m.$$
Similarly,$$e''_{\rho_{k}}=\frac{1}{2m+1}({\bf 1}+\sum\limits_{l=1}^{2m}\cos lk\theta \cdot r^l-\sum\limits_{l=1}^{2m}\sin lk\theta\cdot r^ls),\quad 1\leq k \leq m.$$

(ii) If $n=2m$, the primitive central idempotents of $D_{2n}$ are given by
$$e_{\rho_{k}}=\frac{1}{m}\sum_{l=1}^{2m}\cos lk\theta \cdot r^l,\quad 1\leq k \leq m-1.$$
Therefore,
$$e'_{\rho_{k}}=\frac{1}{2m}({\bf 1}+\sum_{l=1}^{2m-1}\cos lk\theta \cdot r^l+\sum_{l=1}^{2m-1}\sin lk\theta\cdot r^ls),\quad 1\leq k \leq m-1.$$
 Similarly, $$e''_{\rho_{k}}=\frac{1}{2m}({\bf 1}+\sum\limits_{l=1}^{2m-1}\cos lk\theta \cdot r^l-\sum\limits_{l=1}^{2m-1}\sin lk\theta\cdot r^ls),\quad 1\leq k \leq m-1.\qedhere$$
\end{proof}
\begin{exam}\label{exD8}
Let $D_{8}$ be a dihedral group with order $8$. Then $m=2$, $k=1, n=4$, there is a primitive decomposition of idempotents as follows.
\begin{align*}
e_{\rho_{1}}&=\frac{1}{2}({\bf1}-r^2),\\
e'_{\rho_{1}}&=\frac{1}{4}({\bf 1}-r^2+rs-r^3s),\\     e''_{\rho_{1}}&=\frac{1}{4}({\bf 1}-r^2-rs+r^3s).
\end{align*}
\end{exam}
\section{A primitive decomposition of idempotents of $\CC[Q_{4m}]$}\label{se:gqg-idem}
\subsection{Conjugacy classes of $Q_{4m}$}
Let $Q_{4m}$ be the generalized quaternion group of order $4m$, i.e.
$$Q_{4m}=\{a,b\,|\,a^{2m}={\bf 1},\,a^m=b^2,\,b^{-1}ab=a^{-1}\}.$$
$Q_{4m}$ has the following conjugacy classes:
$$\{{\bf1}\},\,\{a^{m}\},\,\{a^{\pm r}\,|\,1\leq r\leq m-1\},\,\{a^{2k}b\,|\,0\leq k \leq m-1\},\,\{a^{2k-1}b\,|\,0\leq k \leq m-1\}.$$

\subsection{Character table of $Q_{4m}$}
The derived subgroup $Q_{4m}^\prime=\langle a^2\rangle$. Indeed, $\langle a^2\rangle$ is a normal subgroup of $Q_{4m}$, and $|Q_{4m}/\langle a^2\rangle|=4$, hence $Q_{4m}/\langle a^2\rangle$ is abelian and $\langle a^2\rangle\supseteq Q_{4m}^\prime$. Clearly, $b^{-1}a^{-1}b=a$, thus $b^{-1}a^{-1}ba=a^2\in Q_{4m}^\prime$, as $\langle a^2\rangle\subseteq Q_{4m}^\prime$.\\
 \indent
 As $|Q_{4m}/\langle a^2\rangle|=4$, $Q_{4m}/\langle a^2\rangle\cong C_{4}$ or $Q_{4m}/\langle a^2\rangle\cong C_{2}\times C_{2}$, and $Q_{4m}$ has four irreducible one-dimensional representations. 
 Also, it has $m-1$ mutually non-equivalent two-dimensional irreducible representations \cite[Exs. 17.6, 18.3, 23.5]{JL}. We recall these two-dimensional irreducible representations of $Q_{4m}$ as follows.

Let $\varepsilon\coloneqq e^{\pi i/m}\in\mathbb{C}$ with $i\coloneqq \sqrt{-1}$. For each $k$ with $1\leq k \leq m-1$, denote matrices
$$A_{k}=\left(
\begin{array}{cc}
\varepsilon^{k} & 0 \\
0 & \varepsilon^{-k} \\
\end{array}
\right),\quad
B_{k}=\left(
\begin{array}{cc}
0 & 1 \\
(-1)^k & 0 \\
\end{array}
\right),$$
which satisfy the following relations:
$$A_{k}^{2m}=I,\quad A_{k}^{m}=B_{k}^{2},\quad B_{k}^{-1}A_{k}B_{k}=A_{k}^{-1}.$$
Hence, it follows that
\begin{equation}\label{eq:gq-2-repn}
\rho_{k}:Q_{4m}\rightarrow GL(2,\mathbb{C})
\end{equation}
defined by
$$a\mapsto A_{k},\quad  b\mapsto B_{k}$$
is a group homomorphism, and we obtain a representation $(\CC^2,\rho_{k})$ of $Q_{4m}$.

(i) When $m$ is odd, as $2\nmid m$, we know that $b^2=a^m\not\in Q_{4m}'$, hence the order of $b$ can not be 2. Then $b$ is of order 4, so $Q_{4m}/\langle a^2\rangle \cong C_{4}$. We set $\vartheta= \dfrac{\pi}{m}$, and list the character table of $Q_{4m}$:
\begin{table}[h]
\begin{tabular}{|c|c|c|c|c|c|c|c|c|}
  \hline
   & {\bf1} & $a$ & $a^2$ & $\cdots$ &$a^{m-1}$&$a^m$& $b$ & $ab$ \\
   & (1) & (2) & (2)& $\cdots$ &(2)&(1) & ($m$) & ($m$) \\
  \hline
  $\chi_{1} $ & 1&1&1&\dots & 1 & 1 & 1 & 1 \\
  $\chi_{2} $ & 1 & 1 & 1 & $\cdots$ & $1$ &$1$& $-1$ & $-1$ \\
  $\chi_{3} $ & 1 & $-1$  & 1& $\cdots$ & $(-1)^{m-1}$ & $-1$ & $i$ & $-i$ \\
 $\chi_{4} $ & 1 & $-1$ & 1 & $\cdots$ & $(-1)^{m-1}$ & $-1$ & $-i$ & $i$ \\
  $\chi_{\rho_{1}} $ & 2  &$2\cos\vartheta$ & $2\cos2\vartheta$& $\cdots$ & $2\cos (m-1)\vartheta$ &$-2$& 0 & 0 \\
  $\vdots$ & $\vdots$ & $\vdots$ & $\vdots$ & $\vdots$ & $\vdots$ & $\vdots$ & $\vdots$ & $\vdots$ \\
  $\chi_{\rho_{m-1}} $ & 2  &$2\cos(m-1)\vartheta$ & $2\cos2(m-1)\vartheta$ & $\cdots$ & $2\cos (m-1)^2\vartheta$ &$2(-1)^{m-1}$& 0 & 0 \\
  \hline
\end{tabular}
\vspace{.5em}
\caption{Irreducible characters of $Q_{4m}$}\label{tb:3}
\end{table}

(ii) When $m$ is even, as $2\,|\,m$, we have $b^2=a^m\in Q_{4m}'$. Therefore, $Q_{4m}/\langle a^2\rangle\cong C_{2}\times C_{2}$. We set $\vartheta= \dfrac{\pi}{m}$, and list the character table of $Q_{4m}$:
\begin{table}[h]
\begin{tabular}{|c|c|c|c|c|c|c|c|c|}
  \hline
   & {\bf1} & $a$ & $a^2$ & $\cdots$ &$a^{m-1}$&$a^m$& $b$ & $ab$ \\
   & (1) & (2) & (2)&$\cdots$ &(2)&(1) & ($m$) & ($m$) \\
  \hline
  $\chi_{1} $ & 1&1&1&$\cdots$ & 1 & 1 & 1 & 1 \\
  $\chi_{2} $ & 1 & 1 &1&$\cdots$ & 1 &1& $-1$ & $-1$ \\
  $\chi_{3} $ & 1 & $-1$ & 1&$\cdots$ &$(-1)^{m-1}$&$1$ & 1 & $-1$ \\
 $\chi_{4} $ & 1 & $-1$&1&$\cdots$  & $(-1)^{m-1}$ &$1$& $-1$ & 1 \\
  $\chi_{\rho_{1}} $ & 2  &$2\cos\vartheta$ & $2\cos2\vartheta$& $\cdots$ & $2\cos (m-1)\vartheta$ &$-2$& 0 & 0 \\
  $\vdots$ & $\vdots$ & $\vdots$ & $\vdots$ & $\vdots$ & $\vdots$ & $\vdots$ & $\vdots$ & $\vdots$ \\
  $\chi_{\rho_{m-1}} $ & 2  &$2\cos(m-1)\vartheta$ & $2\cos2(m-1)\vartheta$ & $\cdots$ & $2\cos (m-1)^2\vartheta$ &$2(-1)^{m-1}$& 0 & 0 \\
  \hline
\end{tabular}
\vspace{.5em}
\caption{Irreducible characters of $Q_{4m}$}\label{tb:4}
\end{table}

\subsection{A primitive decomposition of idempotents}
First applying Eq.~\eqref{eq:pcidem} to the case of generalized quaternion group $Q_{4m}$, we have
\begin{prop}\label{prop:gqg-idem}
Let $Q_{4m}$ be the generalized quaternion group of order $4m$. The primitive central idempotents corresponding to the one-dimensional irreducible representations of $Q_{4m}$ are as follows.
\begin{enumerate}[(i)]
\item
When $m$ is odd,
\begin{align*}
e_{1}&=\frac{1}{4m}(\sum\limits_{l=1}^{2m} a^l+\sum\limits_{l=1}^{2m} a^lb),\\
e_{2}&=\frac{1}{4m}(\sum\limits_{l=1}^{2m} a^l-\sum\limits_{l=1}^{2m} a^lb),\\
e_{3}&=\frac{1}{4m}[{\bf 1}+i\sum_{l=1}^{2m}(-1)^l\cdot a^lb+\sum_{l=1}^{m-1}(-1)^l\cdot (a^l+a^{-l})- a^m],\\
e_{4}&=\frac{1}{4m}[{\bf 1}+i\sum_{l=1}^{2m}(-1)^{l+1}\cdot a^lb+\sum_{l=1}^{m-1}(-1)^l\cdot (a^l+a^{-l})- a^m].
\end{align*}
\item
When $m$ is even,
\begin{align*}
e_{1}&=\frac{1}{4m}(\sum\limits_{l=1}^{2m} a^l+\sum\limits_{l=1}^{2m} a^lb),\\
e_{2}&=\frac{1}{4m}(\sum\limits_{l=1}^{2m} a^l-\sum\limits_{l=1}^{2m} a^lb),\\
e_{3}&=\frac{1}{4m}[{\bf 1}+\sum_{l=1}^{2m}(-1)^l\cdot a^lb+\sum_{l=1}^{m-1}(-1)^l\cdot (a^l+a^{-l})+a^m],\\
e_{4}&=\frac{1}{4m}[{\bf 1}+\sum_{l=1}^{2m}(-1)^{l+1}\cdot a^lb+\sum_{l=1}^{m-1}(-1)^l\cdot (a^l+a^{-l})+a^m].
\end{align*}
\end{enumerate}
\end{prop}

For other primitive idempotents corresponding to two-dimensional irreducible representations of $Q_{4m}$, we have
\begin{theorem}\label{th:gqg-idem}
Let $Q_{4m}$ be the generalized quaternion group of order $4m$. Then 
we have the following primitive decomposition $e_{\rho_{k}}=e'_{\rho_{k}}+e''_{\rho_{k}}$ of the primitive central idempotent $e_{\rho_{k}}$ corresponding to the two-dimensional irreducible representation $(\CC^2,\rho_k)$ of $Q_{4m}$ defined in Eq.~\eqref{eq:gq-2-repn} for $k=1,\dots,m-1$.
\begin{enumerate}[(i)]
\item
When $k$ is odd,
\begin{align*}
e_{\rho_{k}}&=\frac{1}{m}\sum_{l=1}^{2m}\cos lk\vartheta\cdot a^l,\\
e'_{\rho_{k}}&=-\frac{1}{2mi\sin k\vartheta}\sum_{l=1}^{2m}(\varepsilon^{k}a^{m+l}-a^{m+l-1})\cos lk\vartheta,\\
e''_{\rho_{k}}&=-\frac{1}{2mi\sin k\vartheta}\sum_{l=1}^{2m}(a^{m+l-1}-\varepsilon^{-k}a^{m+l})\cos lk\vartheta,
\end{align*}
with $\vartheta = \dfrac {\pi}{m}$ and $1\leq k \leq m-1$;
\item
When $k$ is even,
\begin{align*}
e_{\rho_{k}}&=\frac{1}{m}\sum_{l=1}^{2m}\cos lk\vartheta\cdot a^l,\\
e'_{\rho_{k}}&=\frac{1}{2mi\sin k\vartheta}\sum_{l=1}^{2m}(\varepsilon^{k}a^{m+l}-a^{m+l-1})\cos lk\vartheta,\\
e''_{\rho_{k}}&= \frac{1}{2mi\sin k\vartheta}\sum\limits_{l=1}^{2m}(a^{m+l-1}-\varepsilon^{-k}a^{m+l})\cos lk\vartheta,
\end{align*}
with $\vartheta = \dfrac {\pi}{m}$ and $1\leq k \leq m-1$.
\end{enumerate}
\end{theorem}
\begin{proof}
(i) When $k$ is odd, under the group homomorphism $\rho_k:Q_{4m}\to GL(2,\CC)$, we have
$$a \mapsto \begin{pmatrix}
      \varepsilon^{k} &0\\
      0& \varepsilon^{-k}
     \end{pmatrix},\quad
      b \mapsto \begin{pmatrix}
      0 &1\\
      -1&0
      \end{pmatrix}.$$
Then
$$ab\mapsto \begin{pmatrix}
      0 &\varepsilon^{k}\\
      -\varepsilon^{-k}&0
      \end{pmatrix},\quad
      \varepsilon^{k}b \mapsto \begin{pmatrix}
      0 &\varepsilon^{k}\\
      -\varepsilon^{k}&0
      \end{pmatrix},\quad
      \varepsilon^{-k}b \mapsto \begin{pmatrix}
      0 &\varepsilon^{-k}\\
      -\varepsilon^{-k}&0
      \end{pmatrix}.$$
Therefore,
$$\varepsilon^{k}b-ab\mapsto\begin{pmatrix}
        0&0\\
      \varepsilon^{-k}-\varepsilon^{k}&0
      \end{pmatrix},\quad
      \varepsilon^{-k}b-ab\mapsto\begin{pmatrix}
        0&\varepsilon^{-k}-\varepsilon^{k}\\
      0&0
      \end{pmatrix}.$$
As $\varepsilon^{-k}-\varepsilon^{k}\neq0$,
it implies that
\begin{align*}
\frac{1}{\varepsilon^{-k}-\varepsilon^{k}}(\varepsilon^{k}b-ab)&\mapsto\begin{pmatrix}
        0&0\\
      1&0
      \end{pmatrix},\\
\frac{1}{\varepsilon^{-k}-\varepsilon^{k}}(\varepsilon^{-k}b-ab)&\mapsto\begin{pmatrix}
        0&1\\
      0&0
      \end{pmatrix}.
\end{align*}
As a result, we have
\begin{align*}
\frac{b}{\varepsilon^{-k}-\varepsilon^{k}}(\varepsilon^{k}b-ab)&\mapsto\begin{pmatrix}
        0&0\\
      0&1
      \end{pmatrix},\\
-\frac{b}{\varepsilon^{-k}-\varepsilon^{k}}(\varepsilon^{-k}b-ab)&\mapsto\begin{pmatrix}
        1&0\\
      0&0
      \end{pmatrix}.
\end{align*}
Now one can compute the primitive central idempotents $e_{\rho_{k}}$ via the character table, and then get their desired primitive decompositions by the similar argument as in the proof of Theorem \ref{th:di-idem}.

(ii) When $k$ is even, by similar arguments as in the case when $k$ is odd.
\end{proof}

\begin{exam}\label{exQ8}
Let $Q_{8}$ be a generalized quaternion group of order $8$, then $k=1$, $m=2$. Therefore,
\begin{align*}
e_{\rho_{1}}&=\frac{1}{2}({\bf1}-a^2), \\
e'_{\rho_{1}}&=-\frac{1}{4i}(a^3+i\cdot a^2-a-i\cdot{\bf1}),\\
 e''_{\rho_{1}}&=-\frac{1}{4i}(-a^3+i\cdot a^2+a-i\cdot{\bf1}).
\end{align*}
\end{exam}

\section{Trigonometric identities}\label{se:tri-iden}
Here we find the following two sets of trigonometric identities covering the orthogonality relations in the character tables of dihedral groups and generalized quaternion groups.
 \begin{prop}\label{prop4.1}
 For any $n\geq1$ and $1\leq k \leq n-1$, and any angle $\theta$ which is not an integer multiple of\, $2\pi$, we have
 \begin{enumerate}[(i)]
\item
$\displaystyle\sum\limits_{r=0}^{n-1} (-1)^r \cos \dfrac{rk\pi}{n}=
\begin{cases}
 1, & \text{ $n+k$ odd},\\
 0, & \text{ $n+k$ even};
 \end{cases}$
\item
$\displaystyle\sum\limits_{r=1}^{n}  \cos r\theta
=\dfrac{\sin(\frac{\theta}{2}+n\theta)}{2\sin\frac{\theta}{2}}
-\dfrac{1}{2}$.
 \end{enumerate}
 \end{prop}
\begin{proof}
(i) Note that
$$\cos(\frac{rk\pi}{n}\pm \frac{k\pi}{2n})=\cos\frac{rk\pi}{n}\cos\frac{k\pi}{2n}\mp \sin\frac{rk\pi}{n}\sin\frac{k\pi}{2n}$$
imply the following product-to-sum identity
$$\cos\frac{rk\pi}{n}\cos\frac{k\pi}{2n}=\frac{1}{2}[\cos\frac{(2r+1)k\pi}{2n}+\cos\frac{(2r-1)k\pi}{2n}].$$
As a result, we have
\begin{align*}
\sum\limits_{r=0}^{n-1} (-1)^r \cos \frac{rk\pi}{n}\cos\frac{k\pi}{2n}
&=\cos\frac{k\pi}{2n}+\sum\limits_{r=1}^{n-1} (-1)^r\cdot \frac{1}{2}[\cos\frac{(2r+1)k\pi}{2n}+\cos\frac{(2r-1)k\pi}{2n}]\\
&=\cos\frac{k\pi}{2n}+\frac{1}{2}\sum\limits_{r=1}^{n-1}(-1)^{r}\cos\frac{(2r+1)k\pi}{2n}
+\frac{1}{2}\sum\limits_{r=0}^{n-2}(-1)^{r+1}\cos\frac{(2r+1)k\pi}{2n}\\
&=\cos\frac{k\pi}{2n}+\frac{1}{2}(-1)^{n-1}\cos\frac{(2n-1)k\pi}{2n}-\frac{1}{2}\cos\frac{k\pi}{2n}\\
&=\frac{1}{2}\cos\frac{k\pi}{2n}+\frac{1}{2}(-1)^{n-1}\cos(k\pi-\frac{k\pi}{2n})\\
&=\frac{1}{2}[1+(-1)^{n+k-1}]\cos\frac{k\pi}{2n}.
\end{align*}
Since $\cos\dfrac{k\pi}{2n}\neq 0$ for any $1\leq k \leq n-1$, we see that
$$\sum\limits_{r=0}^{n-1} (-1)^r \cos \frac{rk\pi}{n}=\frac{1}{2}[1+(-1)^{n+k-1}]=
\begin{cases}
 1, & \text{ $n+k$ odd},\\
 0, & \text{ $n+k$ even}.
 \end{cases}$$

(ii) Similarly by product-to-sum identities, we see that
\begin{align*}
2\sin \frac{\theta}{2}\sum\limits_{r=1}^{n}\cos r\theta&=2\sin\frac{\theta}{2}\cos \theta+\dots+2\sin\frac{\theta}{2}\cos n\theta\\
&=\sin\frac{3\theta}{2}-\sin\frac{\theta}{2}+\dots+\sin(n\theta+\frac{\theta}{2})-\sin(n\theta-\frac{\theta}{2})\\
&=\sin(n\theta+\frac{\theta}{2})-\sin\frac{\theta}{2}.
\end{align*}
Since $\theta$ is not an integer multiple of $2\pi$, we obtain that
$$\sum\limits_{r=1}^{n}  \cos r\theta=\dfrac{\sin(\frac{\theta}{2}+n\theta)}{2\sin\frac{\theta}{2}}-\frac{1}{2}.\qedhere$$
\end{proof}

Next we clarify how these identities are connected to the character tables of dihedral groups and generalized quaternion groups.
\begin{exam}
Using the first orthogonality relation in the character tables of $D_{2n}$ in Table~\ref{tb:1} when $n=2m+1$ and $\theta=\dfrac{2\pi}{2m+1}$, we have
$$\langle\chi_{1},\chi_{\rho_{k}}\rangle=\frac{1}{4m+2}[2+4\sum\limits_{r=1}^{m}\cos kr\theta]=0,\quad 1\leq k\leq m.$$
The resulting identities
$$\sum\limits_{r=1}^{m}\cos kr\theta=-\frac{1}{2},\quad 1\leq k\leq m,$$
and the identities due to $\langle\chi_{1},\chi_{\rho_{k}}\rangle=0$ in Table~\ref{tb:2} are all special cases of Prop.~\ref{prop4.1} (ii). What's more,
$$\langle\chi_{\rho_a},\chi_{\rho_{b}}\rangle=\frac{1}{4m+2}[4+8\sum\limits_{r=1}^{m}\cos ar\theta\cos br\theta]=0, \quad 1\leq a,b\leq m,\,a\neq b.$$
That is,
$$\sum\limits_{r=1}^{m}\cos ar\theta\cos br\theta=-\frac{1}{2},$$
which can also be deduced by Prop.~\ref{prop4.1} (ii).
\end{exam}

\begin{exam}
Using the first orthogonality relation in the character tables of $Q_{4m}$ in Table~\ref{tb:3} and Table~\ref{tb:4},
when $m$ is odd,$$\langle\chi_{3},\chi_{\rho_{k}}\rangle=\frac{1}{4m}[2+4\sum\limits_{r=1}^{m-1}(-1)^r \cos \dfrac{kr\pi}{m}+2(-1)^{k+1}]=0,\quad 1\leq k\leq m-1.$$
When $m$ is even,$$\langle\chi_{3},\chi_{\rho_{k}}\rangle=\frac{1}{4m}[2+4\sum\limits_{r=1}^{m-1}(-1)^r \cos \dfrac{kr\pi}{m}+2(-1)^{k}]=0,\quad 1\leq k\leq m-1.$$
That means
$$\sum\limits_{r=1}^{m-1} (-1)^r \cos \dfrac{kr\pi}{m}=\begin{cases}
 0, & \text{ $m+k$ odd},\\
 -1, & \text{ $m+k$ even},
 \end{cases}$$
equivalent to Prop~\ref{prop4.1} (i). The identities by $\langle\chi_{3},\chi_{\rho_{k}}\rangle=0$ in Table~\ref{tb:2} are the same.
Also, we have
$$\langle\chi_{\rho_{a}},\chi_{\rho_{b}}\rangle=\frac{1}{4m}[4+8\sum\limits_{r=1}^{m-1}\cos \dfrac{ar\pi}{m}\cos \dfrac{br\pi}{m}+4(-1)^{a+b}]=0,\quad 1\leq a,b\leq m-1,\,a\neq b.$$
That is,
$$\sum\limits_{r=1}^{m-1}\cos \dfrac{ar\pi}{m} \cos \dfrac{br\pi}{m}=\begin{cases}
 0, & \text{ $a+b$ odd},\\
 -1, & \text{ $a+b$ even},
 \end{cases}$$
which can also be deduced by Prop.~\ref{prop4.1} (ii).
\end{exam}

\section{A group algebra isomorphism between $\mathbb{C}[Q_{8}]$ and $\mathbb{C}[D_{8}]$}\label{se:alg-isom}
In this section we would like to specifically describe a group algebra isomorphism between $\mathbb{C}[Q_{8}]$ and $\mathbb{C}[D_{8}]$, offering a correspondence between two complete sets of their primitive orthogonal idempotents given in Prop.~\ref{prop:di-idem}, Theorem~\ref{th:di-idem} and Prop.~\ref{prop:gqg-idem}, Theorem~\ref{th:gqg-idem} respectively.
\begin{theorem}\label{th:alg-isom}
There is an algebra isomorphism
$$\psi: \mathbb{C}[Q_{8}]\rightarrow \mathbb{C}[D_{8}]$$
mapping any
$\alpha=x_{0}\cdot{\bf 1}+x_{1}\cdot a^2+x_{2}\cdot a+x_{3}\cdot a^3+x_{4}\cdot b+x_{5}\cdot a^2b+x_{6}\cdot ab+x_{7}\cdot a^3b$ to
\begin{align*}
\psi(\alpha)&=x_{0}\cdot{\bf1}+x_{7}\cdot r+x_{1}\cdot r^2+x_{6}\cdot r^3+\frac{1}{2}(x_{2}+x_{3}-ix_{4}+ix_{5})\cdot s+\frac{1}{2}(-ix_{2}+ix_{3}+x_{4}+x_{5})\cdot rs\\
&\quad+\frac{1}{2}(x_{2}+x_{3}+ix_{4}-ix_{5})\cdot r^2s+
\frac{1}{2}(ix_{2}-ix_{3}+x_{4}+x_{5})\cdot r^3s,
\end{align*}
with $i\coloneqq \sqrt{-1}$ and $x_{i}\in \mathbb{C}$.
\end{theorem}
\begin{proof}
We note that the generalized quaternion group $Q_{4m}$ and the dihedral group $D_{2n}$ have the same character table when $n=2m$ and $2\,|\,m$. In particular, the smallest case $Q_{8}$ and $D_{8}$ have the same values in the first column, and consequently $\mathbb{C}[Q_{8}]\cong \mathbb{C}[D_{8}]$ as algebras by Lemma~\ref{lem:wed}.

The primitive central idempotents corresponding to the two-dimensional irreducible representations of $\mathbb{C}[Q_{8}]$ and $\mathbb{C}[D_{8}]$ respectively are
$$\frac{1}{2}({\bf1}-a^2),\quad\frac{1}{2}({\bf1}-r^2).$$
Under any algebra isomorphism from $\mathbb{C}[Q_{8}]$ to $\mathbb{C}[D_{8}]$, we must have
$${\bf 1}\mapsto{\bf 1},\quad a^2\mapsto r^2.$$

On the other hand, by Prop.~\ref{prop:gqg-idem} all primitive central idempotents corresponding to the four one-dimensional representations of $\mathbb{C}[Q_{8}]$ are as follows:
$$e_{1}=\frac{1}{8}({\bf1}+a^2+a+a^3+b+a^2b+ab+a^3b),$$
$$e_{2}=\frac{1}{8}({\bf1}+a^2+a+a^3-b-a^2b-ab-a^3b),$$
$$e_{3}=\frac{1}{8}({\bf1}+a^2-a-a^3+b+a^2b-ab-a^3b),$$
$$e_{4}=\frac{1}{8}({\bf1}+a^2-a-a^3-b-a^2b+ab+a^3b).$$
That is,
$$\left(
  \begin{array}{c}
    e_{1} \\
    e_{2} \\
    e_{3} \\
    e_{4}\\
  \end{array}
\right)=\frac{1}{8}\left(
                \begin{array}{cccc}
                  1 &  1 & 1 & 1 \\
                  1 &  1 & -1 & -1 \\
                  1 &  -1 & 1 & -1 \\
                  1 & -1 &  -1 &  1 \\
                \end{array}
              \right)\left(
                        \begin{array}{c}
                          {\bf1}+a^2 \\
                          a+a^3 \\
                          b+a^2b \\
                          ab+a^3b \\
                        \end{array}
                      \right).$$
By Prop.~\ref{prop:di-idem}, all primitive central idempotents corresponding to the four one-dimensional representations of $\mathbb{C}[D_{8}]$ are as follows:
$$e_{1}=\frac{1}{8}({\bf1}+r+r^2+r^3+s+rs+r^2s+r^3s),$$
$$e_{2}=\frac{1}{8}({\bf1}-r+r^2-r^3+s-rs+r^2s-r^3s),$$
$$e_{3}=\frac{1}{8}({\bf1}-r+r^2-r^3-s+rs-r^2s+r^3s),$$
$$e_{4}=\frac{1}{8}({\bf1}+r+r^2+r^3-s-rs-r^2s-r^3s).$$
Namely,
$$\left(
  \begin{array}{c}
    e_{1} \\
    e_{2} \\
    e_{3} \\
    e_{4}\\
  \end{array}
\right)=\frac{1}{8}\left(
                \begin{array}{cccc}
                  1 &  1 & 1 & 1 \\
                  1 &  1 & -1 & -1 \\
                  1 &  -1 & 1 & -1 \\
                  1 & -1 &  -1 &  1 \\
                \end{array}
              \right)\left(
                        \begin{array}{c}
                          {\bf1}+r^2 \\
                          s+r^2s \\
                          rs+r^3s \\
                          r+r^3 \\
                        \end{array}
                      \right).$$
Therefore, we can require that our desired algebra isomorphism $\psi:\mathbb{C}[Q_{8}]\rightarrow\mathbb{C}[D_{8}]$ satisfies
\begin{align*}
ab+a^3b&\mapsto r+r^3,\\
a+a^3&\mapsto s+r^2s,\\
b+a^2b&\mapsto rs+r^3s.
\end{align*}

Furthermore, since $\psi(\xi_{1}\xi_{2})=\psi(\xi_{1})\psi(\xi_{2})$ for any $\xi_{1},\xi_{2}\in \mathbb{C}[Q_{8}]$, the map $\psi$ also satisfies:
\begin{align*}
\psi(abab)&=\psi(b^2)=\psi(a^2)=r^2=\psi(ab)^{2},\\
\psi(ab+a^3b)&=\psi(ab)({\bf 1}+r^2)=r+r^3=r({\bf 1}+r^2).
\end{align*}
That is,
\begin{equation*}
\begin{cases}
\psi(ab)^2=r^2,\\
\psi(ab)-r\in({\bf 1}-r^2),
\end{cases}
\end{equation*}
as the principal ideal $({\bf 1}-r^2)$ is the annihilator of ${\bf 1}+r^2$. Similarly,
\begin{equation*}
\begin{cases}
\psi(a)^2=r^2,\\
\psi(a)-s\in ({\bf 1}-r^2),
\end{cases}\quad
\begin{cases}
\psi(b)^2=r^2,\\
\psi(b)-rs\in ({\bf 1}-r^2).
\end{cases}
\end{equation*}
Therefore, we can set
\begin{equation*}
\begin{cases}
\psi(ab)=r+(k_{1}\cdot {\bf 1}+k_{2}\cdot r+k_{3}\cdot rs+k_{4}\cdot s)({\bf 1}-r^2),\\
\psi(a)=s+(k_{9}\cdot {\bf 1}+k_{10}\cdot r+k_{11}\cdot rs+k_{12}\cdot s)({\bf 1}-r^2),\\
\psi(b)=rs+(k_{5}\cdot {\bf 1}+k_{6}\cdot r+k_{7}\cdot rs+k_{8}\cdot s)({\bf 1}-r^2),
\end{cases}
\end{equation*}
with $k_1,\dots,k_{12}\in \mathbb{C}$, and obtain the following system of equations,
\begin{equation*}
\begin{cases}k_{1}k_{3}=0, k_{1}k_{4}=0,\\4k_{1}k_{2}+2k_{1}=0,\\2k^{2}_{1}+2k^{2}_{3}+2k^{2}_{4}-2k^{2}_{2}-2k_{2}=0;\\
k_{5}k_{6}=0, k_{5}k_{8}=0,\\4k_{5}k_{7}+2k_{5}=0,\\2k^{2}_{5}+2k^{2}_{7}+2k^{2}_{8}-2k^{2}_{6}+1+2k_{7}=0;\\
k_{9}k_{10}=0, k_{9}k_{11}=0,\\4k_{9}k_{12}+2k_{9}=0,\\2k^{2}_{9}+2k^{2}_{11}+2k^{2}_{12}-2k^{2}_{10}+1+2x_{12}=0.
\end{cases}
\end{equation*}
Note that there is more than one solution for this system of equations, and any one of these solutions must also satisfy:
\begin{equation*}
\begin{cases}
\psi(a)\psi(b)=\psi(ab),\\
\psi(b)\psi(ab)=\psi(a),\\
\psi(ab)\psi(a)=\psi(b).
\end{cases}
\end{equation*}
However, these three additional equalities fail to hold simultaneously for any solution in which $k_{1}, k_{5}$ and $k_{9}$ are not all zero. Instead, we find the  solution below satisfying all these equations:
\begin{align*}
&k_{1}=k_{3}=k_{4}=0,\quad k_{2}=-1,\quad k_{5}=k_{6}=0,\\
&k_{7}=-\frac{1}{2},\quad k_{8}=-\frac{i}{2},\quad k_{9}=k_{10}=0,\quad k_{11}=-\frac{i}{2},\quad k_{12}=-\frac{1}{2}.
\end{align*}
That is,
\begin{align*}
ab&\mapsto r^3,\\
a&\mapsto \frac{1}{2}(r^2s+s-i\cdot rs+i\cdot r^3s),\\
b&\mapsto \frac{1}{2}(rs+r^3s-i\cdot s+i\cdot r^2s).
\end{align*}
Then
\begin{align*}
a^3b&\mapsto r,\\
a^2b&\mapsto \frac{1}{2}(rs+r^3s+i\cdot s-i\cdot r^2s),\\
a^3&\mapsto \frac{1}{2}(r^2s+s+i\cdot rs-i\cdot r^3s).
\end{align*}
Now we specifically verify that the stated linear map $\psi:\mathbb{C}[Q_{8}]\rightarrow\mathbb{C}[D_{8}]$ is an algebra isomorphism as desired:
\begin{align*}
       \psi(a)^2&=\frac{1}{4}(r^2s+s-i rs+i r^3s)^2\\
       &=\frac{1}{4}({\bf1}+r^2-ir+ir^3+r^2+{\bf1}-ir^3+ir-ir^3-ir-{\bf1}+r^2+ir+ir^3+r^2-{\bf1})\\
       &=r^2=\psi(a^2);\\
       \psi(a)^3&= \psi(a^2)\psi(a)=r^2 \cdot\frac{1}{2}(r^2s+s-irs+ir^3s)= \frac{1}{2}(r^2s+s+i rs-i r^3s) =\psi(a^3);\\
       \psi(a)^4&= \psi(a^2)^2 = (r^2)^2=r^4={\bf1};\\
       \psi(a)\psi(b)&=\frac{1}{4}(r^2s+s-irs+ir^3s)(rs+r^3s-is+ir^2s)\\
       &=\frac{1}{4}(r+r^3-ir^2+i{\bf1}+r^3+r-i{\bf1}+ir^2-i{\bf1}-ir^2-r+r^3+ir^2+i{\bf1}+r^3-r)\\
       &=r^3=\psi(ab);\\
       \psi(a)^2\psi(b)&=\psi(a^2)\psi(b)=r^2\cdot\frac{1}{2}(rs+r^3s-is+ir^2s)=\frac{1}{2}(r^3s+rs-ir^2s+is)=\psi(a^2b);\\
       \psi(a)^3\psi(b)&=\psi(a^2)\psi(ab)= r^2\cdot r^3= r =\psi(a^3b);\\
       \psi(b)^2&= \frac{1}{4}(rs+r^3s-is+ir^2s)^2\\
       &=\frac{1}{4}({\bf1}+r^2-ir+ir^3+r^2+{\bf1}-ir^3+ir-ir^3-ir-{\bf1}+r^2+ir+ir^3+r^2-{\bf1})\\
       &=r^2=\psi(a^2);\\
       \psi(a)\psi(b)\psi(a)&= \psi(ab)\psi(a) = r^3\cdot\frac{1}{2}(r^2s+s-i rs+i r^3s) = \frac{1}{2}(rs+r^3s-is+ir^2s) = \psi(b).\qedhere
\end{align*}
%
\end{proof}

According to Prop.~\ref{prop:di-idem}, Theorem~\ref{th:di-idem} and Prop.~\ref{prop:gqg-idem}, Theorem~\ref{th:gqg-idem}, we have two complete sets of primitive orthogonal idempotents of $\CC[Q_{8}]$ and $\CC[D_{8}]$ respectively.
There are primitive idempotents $e_1,\dots,e_4$ corresponding to one-dimensional irreducible representations of $Q_{8}$, and
Example~\ref{exQ8} has calculated the primitive decomposition of idempotents which is given by $\rho_1$ for the unique two-dimensional irreducible representation of $Q_{8}$ up to equivalence
\begin{align*}
e_{\rho_1}&=\frac{1}{2}({\bf1}-a^2)=e'_{\rho_1}+e''_{\rho_1},\\
e'_{\rho_1}&=-\frac{1}{4i}(a^3+i\cdot a^2-a-i\cdot{\bf1}),\\
e''_{\rho_1}&=-\frac{1}{4i}(-a^3+i\cdot a^2+a-i\cdot{\bf1}).
\end{align*}

There are primitive idempotents $\bar{e}_1,\dots,\bar{e}_4$ corresponding to one-dimensional irreducible representations of $D_{8}$, and we see Example~\ref{exD8} know that
the unique two-dimensional irreducible representation $\rho_1$ of $D_{8}$ up to equivalence provides
\begin{align*}
\bar{e}_{\rho_1}&=\frac{1}{2}({\bf1}-r^2)=\bar{e}'_{\rho_1}+\bar{e}''_{\rho_1},\\
\bar{e}'_{\rho_1}&=\frac{1}{4}({\bf 1}-r^2+rs-r^3s),\\
\bar{e}''_{\rho_1}&=\frac{1}{4}({\bf 1}-r^2-rs+r^3s).
\end{align*}
Here we use bar notation to distinguish the complete set of primitive orthogonal idempotents of $\CC[Q_{8}]$ from that of $\CC[D_{8}]$.

The proof of Theorem~\ref{th:alg-isom} has shown that $\psi(e_i)=\bar{e}_i$ for $1\leq i\leq 4$. Now we further check that
\begin{align*}
           \psi(e'_{\rho_1}) & =\psi\left(-\frac{1}{4i}(a^3+i\cdot a^2-a-i\cdot{\bf1})\right) \\
           &=\frac{1}{4}({\bf 1}-r^2-rs+r^3s)=\bar{e}''_{\rho_1},\\
           \psi(e''_{\rho_1}) & =\psi\left(-\frac{1}{4i}(-a^3+i\cdot a^2+a-i\cdot{\bf1})\right)\\
           &=\frac{1}{4}({\bf 1}-r^2+rs-r^3s)=\bar{e}'_{\rho_1}.
         \end{align*}

\noindent
{\bf Question.} In general, we wonder how to find algebra isomorphisms between $\mathbb{C}[Q_{4m}]$ and $\mathbb{C}[D_{2n}]$ when $n=2m$ and $2\,|\,m$, making a one-to-one correspondence between the two complete sets of their primitive orthogonal idempotents given in this paper.

\smallskip\noindent
{\bf Acknowledgments.}
This work is supported by National Natural Science Foundation of China (Grant No. 12071094). We would like to thank professor Victor Bovdi for helpful comments.

\bibliographystyle{amsplain}

\begin{thebibliography}{99}

\bibitem {B1}  S. D. Berman, Group algebras of abelian extensions of finite groups, (Russian) \emph{Dokl. Akad. Nauk
SSSR (N.S.)} {\bf102} (1955),431--434.

\bibitem {B2}  S. D. Berman, Representations of finite groups over an arbitrary field and over rings of integers, \emph{Izv. Akad. Nauk SSSR Ser. Mat.} {\bf 30} (1966), 69--132.

\bibitem{BKP} G. K. Bakshi, R. S. Kulkarni and I. B. S. Passi, The rational group algebra of a finite group, \emph{J. Algebra Appl.} {\bf12} (2013).

\bibitem{CB}C. Bagi\'nski, Modular group algebras of 2-groups of maximal class, \emph{Comm. Algebra} {\bf20} (1992), 1229--1241.


\bibitem{Co} D. B. Coleman, Finite groups with isomorphic group algebras, \emph{Trans. Amer. Math. Soc.} {\bf105} (1962), 1--8.

\bibitem{FF} B. L. Macedo Ferreira, R. Nascimento Ferreira, The Wedderburn $b$-decomposition for alternative baric algebras. \emph{Mitt. Math. Ges. Hamburg.} {\bf37} (2017), 13--25.

\bibitem{GY} Y. Gao, Q. Yue, Idempotents of generalized quaternion group algebras and their applications, \emph{Discrete Mathematics}. {\bf344} (2021).


\bibitem{He} A. Herman, On the automorphism groups of rational group algebras of metacyclic groups, \emph{Comm. Algebra} {\bf25} (1997), 2085--2097.

\bibitem{JL} G. D. James, M. W. Liebeak, Representations and characters of groups. 2 Eds. New York: Cambridge University, 2001.

\bibitem{JeL} E. Jespers, G. Leal, Generators of large subgroups of the unit group of integral group rings, \emph{ Manuscripta Math.} {\bf78} (1993), 303--315.


\bibitem{JR1} E. Jespers, \'A. del R\'io, A structure theorem for the unit group of the integral group ring of some finite groups, \emph{J. Reine Angew. Math.} {\bf521} (2000), 99--117.

\bibitem{JR2} E. Jespers, \'A. del R\'io, Group Ring Groups. Vol. 1. Orders and Generic Constructions of Units, De Gruyter Graduate, De Gruyter, Berlin , 2016.

\bibitem{JOR} E. Jespers, G. Olteanu, \'A. del R\'io, Rational group algebras of finite groups: from idempotents to units of integral group rings, \emph{Algebr. Represent. Theory} {\bf 15} (2012), 359--377.

\bibitem{JORV} E. Jespers, G. Olteanu, \'A. del R\'io, and I. Van Gelder, Group rings of finite strongly monomial group: central units and primitive idempotents. \emph{J. Algebra} {\bf387} (2013), 99--116.

\bibitem{ILPSSZ}I. M. Isaacs, A. I. Lichtman, D. S. Passman, S.K. Sehgal, N. J. A. Sloane, H. J. Zassenhaus, Representation Theory, Group Rings, and Coding Theory: Papers in Honor of S. D. Berman. Contemporary Mathematics, vol. {\bf 93}. American Mathematical Society, Providence, 1989.


\bibitem{Ma} F. E. Brochero Mart\'inez, Structure of finite dihedral group algebra, \emph{Finite Fields Appl.} {\bf5} (2015), 204--214.


\bibitem{OG} G. Olteanu, I. Van Gelder, Finite group algebras of nilpotent groups: A complete set of orthogonal primitive idempotents, \emph{Finite Fields Appl.} {\bf 17} (2011), 157--165.

\bibitem{ORS1} A. Olivieri, \'A. del R\'io and J. J. Sim\'on, On monomial characters and central idempotents of rational group algebras, \emph{Comm. Algebra} {\bf32} (2004), 1531--1550.

\bibitem{ORS2} A. Olivieri, \'A. del R\'io and J. J. Sim\'on, The group of automorphisms of the rational group algebra of a finite metacyclic group, \emph{Comm. Algebra} {\bf34} (2006), 3543--3567.

\bibitem{RS} J. Ritter, S. K. Sehgal, Construction of units in integral group rings of finite nilpotent groups, \emph{Trans. Amer. Math. Soc.} {\bf324} (1991), 603--621.

\bibitem{S} J.-P. Serre, Linear Representations of Finite Groups, Springer, New York, 1977.

\bibitem{TY} D. Tambara, S. Yamagami, Tensor categories with fusion rules of self-duality for finite abelian groups. \emph{J. Algebra} {\bf209} (1998), 692--707.

\bibitem{Ve} C. R. Giraldo Vergara, Wedderburn decomposition of small rational group algebras, in: Groups, Rings and Group Rings, in: \emph{Lect. Notes Pure Appl. Math.} {\bf248} (2006), 191--200.

\bibitem{VM} C. R. Giraldo Vergara, F. E. Brochero Mart\'inez, Wedderburn decomposition of some special rational group algebras, \emph{Lect. Mat.} {\bf23} (2002), 99--106.

\bibitem{We} P. Webb, A Course in Finite Group Representation Theory, Cambridge Studies in Advanced Mathematics, {\bf161}, Cambridge University Press, Cambridge, 2016.


\end{thebibliography}

\end{document}